\documentclass[reqno]{amsart}

\usepackage{amsmath}
\usepackage{amssymb}
\usepackage[margin=1in]{geometry}
\usepackage{xcolor}
\usepackage{comment}
\usepackage{graphicx}
\usepackage{empheq}

\theoremstyle{plain}
\newtheorem{thm}{Theorem}[section]
\newtheorem{lem}{Lemma}[section]
\newtheorem{prop}{Proposition}[section]

\theoremstyle{definition}
\newtheorem{defn}{Definition}[section]

\theoremstyle{remark}
\newtheorem{rmk}{Remark}[section]

\newcommand{\norm}[1]{\left\lVert#1\right\rVert}
\newcommand{\Lip}{\operatorname{Lip}}

\usepackage{cite}
\usepackage{hyperref}
\hypersetup{ colorlinks = true, urlcolor = blue, linkcolor = blue, citecolor = red }
\usepackage{enumitem}
\usepackage{nameref}
\makeatletter
\let\orgdescriptionlabel\descriptionlabel
\renewcommand*{\descriptionlabel}[1]{%
  \let\orglabel\label
  \let\label\@gobble
  \phantomsection
  \edef\@currentlabel{#1\unskip}%
  \let\label\orglabel
  \orgdescriptionlabel{#1}%
}
\numberwithin{equation}{section}

\begin{document}
	
\title[ Lipschitz regularity for fully nonlinear elliptic equations with $(p,q)$-growth]{Lipschitz regularity for fully nonlinear elliptic equations with $(p,q)$-growth}
%%%%% AUTHORS %%%%%

\author[Byun]{Sun-Sig Byun}
\address{Department of Mathematical Sciences and Research Institute of Mathematics,
	Seoul National University, Seoul 08826, Republic of Korea}
\email{byun@snu.ac.kr}

\author[Kim]{Hongsoo Kim}
\address{Department of Mathematical Sciences, Seoul National University, Seoul 08826, Republic of Korea}
\email{rlaghdtn98@snu.ac.kr}

\makeatletter
\@namedef{subjclassname@2020}{\textup{2020} Mathematics Subject Classification}
\makeatother
% \date{\today}
% It is required to enter 2010 MSC.
\subjclass[2020]{35B65, 35D40, 35J15, 35J25}
% Please provide minimum  5 keywords.
\keywords{Fully nonlinear elliptic equations, (p,q)-growth, Lipschitz regularity}

\everymath{\displaystyle}

\begin{abstract}
	We prove the interior and global Lipschitz regularity results for a solution of fully nonlinear equations with $(p,q)$-growth.
    We prove that for a small gap $q-p$, a solution is locally or globally Lipschitz continuous.
    We also prove that a given H\"older continuous solution is Lipschitz continuous under improved bounds for the gap.
    These gap conditions are similar to those required for the regularity of double phase problems in divergence form as in \cite{Mingione152, Mingione18}.
\end{abstract}

\maketitle

\section{Introduction} \label{sec1}
In this paper we study interior and global Lipschitz regularity of viscosity solutions of the following fully nonlinear elliptic equation,
\begin{align} \label{PDE}
    F(D^2u, Du, x) = f(x,Du) \quad \text{in } \Omega,
\end{align}
where $F=F(M,z,x)$ is nonuniformly elliptic with $(p,q)$-growth condition
\begin{align} \label{pqgrowth}
    \lambda |z|^p \norm{N} \leq F(M+N,z,x)- F(M,z,x) \leq \Lambda(|z|^p+|z|^q)\norm{N}
\end{align}
for any $N\geq 0$ with $0<\lambda\leq \Lambda$, $-1<p\leq q$ and a small gap $q-p$, together with the structural assumptions \ref{A2}-\ref{A4}.
In this equation, the ellipticity ratio \eqref{ratio} may blow up when the gradient $|Du|$ is large. Thus, Lipschitz regularity yields a bounded ellipticity ratio and places the equation
in an effectively uniformly elliptic regime.

Some examples of equations satisfying the assumptions are:
\begin{enumerate}
    \item A double phase problem. 
    \begin{align} \label{dp}
        |Du|^p F(D^2u) + a(x)|Du|^q G(D^2u)= f(x),
    \end{align}
    where $F,G$ is uniformly elliptic and $a(\cdot) \in C^\alpha$ with $a(\cdot)\geq0$.
    \item An anisotropic $(p,q)$-growth problem. 
    \begin{align} \label{pq}
      |Du|^p F(D^2u) + \sum_i a_i(x)|D_iu|^{q_i} G_i(D^2u) =f(x),
    \end{align}
    where $q_i\geq 0$, $p\leq q_i$, $q=\max q_i$, $a_i(\cdot) \in C^\alpha$ with $a_i(\cdot)\geq0$, $F$ is uniformly elliptic, and that $G_i$ is degenerate elliptic in the sense that
    \begin{align*}
        0\leq G_i(M+N)-G_i(M) \leq \Lambda\norm{N}
    \end{align*}
    for any $M,N \in S(n)$ with $N\geq 0 $.
\end{enumerate}
The first example is fully nonlinear analogue of the double phase problem
\begin{align} \label{dpc}
    w\rightarrow \int_\Omega |Dw|^p + a(x)|Dw|^q dx,
\end{align}
where $1<p\leq q$ and $a(\cdot) \in C^\alpha$ with $a(\cdot)\geq0$.
The second example is nondivergent counterpart of the following anisotropic problem
\begin{align}
    w\rightarrow \int_\Omega |Dw|^p + \sum _{i} a_i(x)|D_iw|^{q_i} dx.
\end{align}

Equations of this type are said to satisfy a $(p,q)$-growth condition.
From a variational point of view, the regularity of elliptic equations with $(p,q)$-growth condition has been intensively studied since Marcellini's pioneering works \cite{Marcellini89, Marcellini91}. The $(p,q)$-growth problem is about minimizers of the functional
\begin{align} \label{pqproblem}
     w\rightarrow \int_\Omega G(x,Dw) dx, \quad |z|^p \lesssim G(x,z) \lesssim |z|^q+1
\end{align}
with $1<p\le q$.
Then the problem is nonuniformly elliptic since the ellipticity ratio of $G$, which is given by
\begin{align*}
    \mathcal{R}_G(z) := \frac{\text{highest eigenvalue of } \partial_{zz}G(z)}{\text{lowest eigenvalue of } \partial_{zz}G(z)},
\end{align*}
behaves with the polynomial growth as
\begin{align*}
    \mathcal{R}_G(z) \lesssim 1+|z|^{q-p},
\end{align*}
causing the ratio to blow up when $z \rightarrow \infty$.
When the difference between $p$ and $q$ is large, irregular behavior
may occur in the variational setting; see \cite{Giaquinta87,Marcellini91}.

A particularly influential case of $(p,q)$-growth problem is the double phase problem \eqref{dpc} first introduced by Zhikov \cite{Zhikov95, Zhikov97}.
Regularity theory for the double phase problem \eqref{dpc} shows that the admissible gap between $p$ and $q$ depends on the H\"older continuity of $a(\cdot) \in C^\alpha$ and on the \emph{a priori} regularity of minimizers. In particular, sharp gap bounds are known in:
\begin{empheq}[left={\empheqlbrace}]{align}
    \frac{q}{p} \leq 1+\frac{\alpha}{n} \qquad &\text{ if } u \in W^{1,p} \text{ in \cite{Mingione15}} , \label{gap1}  \\
    q-p \leq \alpha \qquad &\text{ if } u \in L^\infty \cap W^{1,p} \text{ in \cite{Mingione152}}, \label{gap2} \\
    q-p < \frac{\alpha}{1-\beta} \qquad &\text{ if } u \in C^\beta \cap W^{1,p} \text{ in \cite{Mingione18}}. \label{gap3}
\end{empheq}
If these bounds are violated, then irregular counterexamples may occur; see \cite{Mingione04, Balci20}.
Recently, these results are extended to cases with orlicz growth conditions, as in \cite{Sumiya25,Ok22,Oh20}.
For the general $(p,q)$-growth problem \eqref{pqproblem}, De Filippis and Mingione \cite{Mingione23, Mingione24} proved the nonuniformly elliptic Schauder regularity results for $u\in W^{1,p}$ under the same threshold as \eqref{gap1}.
See \cite{Adimurthi22,Choe92,Giovanni14, Mingione21} for more results in variational setting.

In contrast to the variational theory, Lipschitz regularity for nonuniformly elliptic non-divergence equations has not yet been investigated to the same extent.
An example of a different class of fully nonlinear equations related to double phase growth is the product-type model introduced by De Filippis \cite{DeFilippis21}:
\begin{align} \label{dpx}
    (|Du|^p +a(x)|Du|^q)F(D^2u) = f(x),
\end{align}
where $F$ is uniformly elliptic and $a(x) \geq 0$ is continuous.
In this setting, the regularity mechanism is fundamentally different from that of the present paper:
thanks to the product structure, solutions of \eqref{dpx} enjoy $C^{1,\alpha}$-regularity for some $\alpha>0$, independently of the gap between $p$ and $q$, without requiring
H\"older regularity of $a(x)$.
When $f\equiv0$, the cutting lemma of Imbert-Silvestre \cite{Imbert13} allows one to remove the `gradient' term and reduce to the uniformly elliptic equation $F(D^2u)=0$, yielding $C^{1,\alpha}$ regularity.
Moreover, a solution satisfies $F(D^2u) = \frac{f}{|Du|^p +a(x)|Du|^q} \in L^\infty$ when $|Du|>1$, which is uniformly elliptic, and so we may use the Ishii-Lions method or \cite{Imbert16,Mooney15} to get the regularity results; see more results in \cite{DaSilva22,Baasandorj24,Andrade22} and the references therein.
In contrast, our model \eqref{dp} is not reducible to a product structure when $F\neq G$, and the gap conditions on $p$ and $q$ become essential.

We also would like to mention the result of the interior Lipschitz regularity of viscosity solutions of anisotropic $\vec{p}$-laplacian equation by by Demengel \cite{Demengel17}
\begin{align*}
    \sum_{i} a_i(x)\partial_i(|\partial_iu|^{p_i-2}\partial_i u) =f.
\end{align*}
This equation is nonuniformly elliptic and the admissible gap condition $q-p<\alpha$ with $p=\min p_i$ and $q=\max p_i$ is comparable to that in the present work.

In this paper we establish interior and global Lipschitz bounds for viscosity solutions of $(p,q)$-growth problem \eqref{PDE} under some proper gap conditions analogous to \eqref{gap2}-\eqref{gap3}. We now state the main assumptions on \eqref{PDE} as follows.
For any $x,y \in \Omega$, $z,w \in \mathbb{R}^n$ with $|z|,|w| \geq C_0$ for some constant $C_0 > 0$, and $M,N \in S(n)$, we assume
\begin{description}
    \item[(A1)\label{A1}] $F$ is of $(p,q)$-growth for some $-1< p\leq q$ in the sense that there holds
    \begin{align*}
        \mathcal{M}^-_{\lambda|z|^p,\Lambda(|z|^p + |z|^q)}(N) \leq F(M+N,z,x) - F(M,z,x) \leq \mathcal{M}^+_{\lambda|z|^p,\Lambda(|z|^p + |z|^q)}(N)
    \end{align*}
    with some constants $0<\lambda< \Lambda$.
    \item[(A2)\label{A2}] $F$ is Lipschitz with respect to the gradient variable $z$ in the sense that
    \begin{align*}
    |F(M,z,x) - F(M,w,x)| \leq \Lambda|z-w|(|z|^{q-1}+|w|^{q-1})(\norm{M}+1),
    \end{align*} 
    for any $ \frac{1}{2}|w|\leq |z| \leq 2|w|$.
    \item[(A3)\label{A3}] $F$ is $\alpha$-H{\"o}lder in $x$-dependence in the sense that
    \begin{align*}
    |F(M,z,x) - F(M,z,y)| \leq \Lambda|x-y|^\alpha|z|^{q}(\norm{M}+1).
    \end{align*}
    \item[(A4)\label{A4}] $F(0,z,x) =0$ and $f \in C(\Omega \times \mathbb{R}^n)$ with
    \begin{align*}
        |f(x,z)| \leq C_f(1+|z|^{q+1}).
    \end{align*}
    for some constant $C_f > 0$.
\end{description}
Observe that \eqref{pqgrowth} is equivalent to the assumption \ref{A1}.
In particular,  \ref{A1} implies the ellipticity ratio is
\begin{align} \label{ratio}
    \mathcal{R}_F(Du) = \frac{\Lambda(|Du|^p +|Du|^q)}{\lambda|Du|^p} = \frac{\Lambda}{\lambda}(1+|Du|^{q-p}),
\end{align}
and hence it may blow up when $|Du|\to\infty$.
If the Lipschitz regularity is proved, then the ellipticity ratio does not blow up and the equation can be considered as the standard $p$-growth problem.

Note that the assumption $F(0,p,x)=0$ does not impose new restriction, as \eqref{PDE} can be written as 
\begin{align*}
    \tilde{F}(D^2u,Du,x):= F(D^2u,Du,x)-F(0,Du,x) =f(x,Du)-F(0,Du,x)=:\tilde{f}(x,Du),
\end{align*}
and then $\tilde{F}$ and $\tilde{f}$ satisfy the same structural assumptions.

We state the main theorem below.
For the interior case, we have the following results.
\begin{thm} \label{Main1}
    Let $u \in C(B_1)$ be a viscosity solution of the $(p,q)$-growth problem
    \begin{align*}
        F(D^2u,Du,x) = f(x,Du) \quad \text{in } \ B_1
    \end{align*}
    under the assumptions \ref{A1}-\ref{A4}. If
    \begin{align} \label{gapl}
        q-p< \alpha,
    \end{align}
    then $u \in \Lip(B_{1/2})$ with
    \begin{align*}
        \norm{u}_{Lip(B_{1/2})} \leq C(\norm{u}_{L^\infty(B_1)}+1)^\theta,
    \end{align*}
    where $\theta = \theta(q-p,\alpha)\geq1$ and $C=C(n,\lambda,\Lambda,p,q,\alpha,C_0,C_f)>0$.\\
    Moreover, if $u \in C^\beta(B_1)$ a priori and
    \begin{align} \label{gapbl}
       q-p< \min \left\{1 + \frac{\beta}{2(1-\beta)}, \frac{\alpha}{1-\beta} \right\},
    \end{align}
    then 
    $u \in \Lip(B_{1/2})$ with
    \begin{align*}
        \norm{u}_{Lip(B_{1/2})} \leq C(\norm{u}_{C^\beta(B_1)}+1)^\theta,
    \end{align*}
    where $\theta = \theta(q-p,\alpha,\beta)\geq1$ and  $C=C(n,\lambda,\Lambda,p,q,\alpha,\beta,C_0,C_f)>0$.
\end{thm}

We also prove global Lipschitz regularity results up to the boundary.
\begin{thm} \label{Main2}
    Let $\Omega \subset \mathbb{R}^n$ be a bounded $C^{1,1}$ domain.
    Let $u \in C(\overline{\Omega})$ be a viscosity solution of the $(p,q)$-growth problem
    \begin{align*}
    \begin{cases}
        F(D^2u,Du,x) = f(x,Du) \quad &\text{in } \ \Omega \\
        u=g\quad &\text{on } \ \partial\Omega
    \end{cases}
    \end{align*}
    under the assumptions \ref{A1}, \ref{A3} and \ref{A4} with $g\in C^{1,1}(\partial\Omega)$. If
    \begin{align} \label{gapg}
        q-p< \alpha,
    \end{align}
   then $u \in \Lip(\overline{\Omega})$ with
   \begin{align*}
    \norm{u}_{Lip(\overline{\Omega})} \leq C(\norm{u}_{L^\infty(\overline{\Omega})}+1)^\theta,
   \end{align*}
   where $\theta = \theta(q-p,\alpha)\geq1$ and $C=C(n,\lambda,\Lambda,p,q,\alpha,C_0 ,C_f,\norm{g}_{C^{1,1}(\partial\Omega)}, \Omega)>0$.\\
   Moreover, if $u \in C^\beta(\overline{\Omega})$ a priori and
    \begin{align} \label{gapbg}
       q-p< \frac{\alpha}{1-\beta},
    \end{align}
    then $u \in \Lip(\overline{\Omega})$ with
   \begin{align*}
    \norm{u}_{Lip(\overline{\Omega})} \leq C(\norm{u}_{C^\beta(\overline{\Omega})}+1)^\theta,
   \end{align*}
   where $\theta = \theta(q-p,\alpha,\beta)\geq1$ and $C=C(n,\lambda,\Lambda,p,q,\alpha,\beta,C_0,C_f ,\norm{g}_{C^{1,1}(\partial\Omega)}, \Omega)>0$.
\end{thm}
\begin{rmk}
We give some remarks on the above theorems.
\begin{enumerate}
\item Assumption \ref{A2} is used only in the \emph{interior} Ishii--Lions argument, where localization terms  $\tfrac{L_2}{2}|x-x_0|^2 $ and $\tfrac{L_2}{2}|y-x_0|^2 $ (see \eqref{intish}) generate a difference between the gradients at two points and one must control the corresponding term in the equation. In the \emph{global} argument (see \eqref{gloish}), such localization terms are absent, and therefore the gradient-difference term does not appear; for this reason \ref{A2} is not required in Theorem~\ref{Main2}.
\item For $u\in L^\infty$, the interior gap \eqref{gapl} and the global gap \eqref{gapg} both match the optimal gap for variational problem \eqref{gap2}.
For $u\in C^\beta$, the global bound \eqref{gapbg} matches \eqref{gap3}, while the interior bound \eqref{gapbl} is slightly more restrictive; this loss is due to the localization terms in the interior Ishii--Lions method.
 \item The two gaps, \eqref{gapbl} and \eqref{gapbg}, for $u \in C^\beta$ can be interpreted as interpolative bounds. 
 Observe that for both bounds, if $\beta \rightarrow 0$, then they reduce to $q-p<\alpha$, while if $\beta \rightarrow 1$, then neither of them is necessary.  
 \item While irregular counterexamples beyond the optimal gap are well known in the divergence-form setting \cite{Mingione04, Balci20}, to the best of our knowledge no analogous irregular counterexamples are currently available for fully nonlinear non-divergence equations with $(p,q)$-growth.
\end{enumerate}
\end{rmk}

Now we explain our approach to the proof of the main theorems.
First, by dividing $|Du|^{p}$ by the $(p,q)$-growth problem \eqref{PDE}, we can simply change the problem into a $(0,\gamma)$-growth problem where $\gamma=q-p$, but it does not render the equation uniformly elliptic because the ellipticity ratio still depends on $|Du|$.
The core of our proof is the Ishii-Lions method \cite{Ishii92,Ishii90}, first yielding $C^\kappa$ bounds for every $\kappa\in(0,1)$ and then proving $u \in Lip$.
For boundary and global estimates, we combine a barrier construction (following Birindelli--Demengel \cite{Birindelli14}) with the global Ishii--Lions method.

The paper is organized as follows.
In Section \ref{sec2} we introduce the notations and some preliminaries.
In Section \ref{sec3} we prove the interior Lipschitz regularity, Theorem \ref{Main1}.
In Section \ref{sec4} we prove the boundary and global Lipschitz regularity, Theorem \ref{Main2}. 

\section{Notations and Preliminaries} \label{sec2}

Throughout the paper, we write $B_r(x_0) = \{ x \in \mathbb{R}^n : |x-x_0| <r \}$ and $B_r = B_r(0)$.
$S(n)$ denotes the space of symmetric $n \times n$ real matrices and $I$ denotes the identity matrix.
For $a,b>0$, $a\approx b$ means there exists a universal $C>1$ such that $\frac{1}{C}b\leq a \leq Cb$.

Now we recall the definition of the inequalities \eqref{PDE} in the viscosity sense from \cite{Ishii92, Caffarelli95} as follows.
\begin{defn}
Let $f \in C(B_1)$.
We say that $u \in C(\overline{B}_1)$ satisfies 
$$F(D^2u,Du,x) \leq f(x,Du) \quad\text{in}\ B_{1} \quad (\text{resp.} \geq) $$
in the viscosity sense, if for any $x_0 \in B_1$ and test function $\psi \in C^2(B_1)$ such that $u-\psi$ has a local minimum (resp. maximum) at $x_0$, then
$$F(D^2\psi(x_0),D\psi(x_0),x_0)\leq f(x_0,D\psi(x_0)) \quad (\text{resp.} \geq).$$ 
\end{defn}

Also, we state the definition of subjet and superjet introduced in \cite{Ishii92}, which can be used in the definition of viscosity solution.
\begin{defn}
    For any continuous function $u \in C(\Omega)$ and $x \in \Omega$, we define superjet and subjet by
    \begin{align*}
        \mathcal{J}^{2,+}_\Omega u(x) = \left\{ (p,X) \in \mathbb{R}^n \times S(n) : u(x+h) \leq u(x) + \langle p, h \rangle  + \frac{1}{2} \langle Xh, h\rangle +o(h^2), \ \ \forall h \in \mathbb{R}^n \right\}, \\
        \mathcal{J}^{2,-}_\Omega u(x) = \left\{ (p,X) \in \mathbb{R}^n \times S(n) : u(x+h) \geq u(x) + \langle p, h \rangle  + \frac{1}{2} \langle Xh, h\rangle +o(h^2), \ \ \forall h \in \mathbb{R}^n \right\}.
    \end{align*}
    Furthermore, we define the closed superjet and subjet by
    \begin{align*}
        \overline{\mathcal{J}^{2,\pm}_\Omega} u(x) = \left\{ (p,X)  : \exists x_n \in \Omega, \exists (p_n,X_n) \in \mathcal{J}^{2,\pm}_\Omega u(x_n), \ (x_n, u(x_n),p_n,X_n) \rightarrow (x,u(x),p,X)\right\}.
    \end{align*}
\end{defn}
We also recall the definition and some properties of the Pucci operator (see \cite{Caffarelli95}).
\begin{defn}
    For given $0<\lambda \leq \Lambda$, we define the Pucci operators $\mathcal{P}_{\lambda,\Lambda}^{\pm} : S(n) \rightarrow \mathbb{R}$ as follows:
\begin{align*}
    \mathcal{M}_{\lambda,\Lambda}^{+}(M) := \lambda \sum_{e_i(M)<0}e_i(M) + \Lambda \sum_{e_i(M)>0}e_i(M), \\
    \mathcal{M}_{\lambda,\Lambda}^{-}(M) := \Lambda \sum_{e_i(M)<0}e_i(M) + \lambda \sum_{e_i(M)>0}e_i(M),
\end{align*}
where $e_i(M)$'s are the eigenvalues of $M$.
\end{defn}
\begin{prop} 
For any $M,N \in S(n)$, we have
    \begin{enumerate}
        \item For $\lambda' \leq \lambda \leq \Lambda \leq \Lambda'$, $\mathcal{M}_{\lambda,\Lambda}^{+}(M) \leq \mathcal{M}_{\lambda',\Lambda'}^{+}(M)$, $\mathcal{M}_{\lambda,\Lambda}^{-}(M) \geq \mathcal{M}_{\lambda',\Lambda'}^{-}(M)$.\\
        \item For $\alpha >0$, $\mathcal{M}_{\alpha\lambda,\alpha\Lambda}^{\pm}(M) = \alpha\mathcal{M}_{\lambda,\Lambda}^{\pm}(M)$. \\
        \item $ \mathcal{M}_{\lambda,\Lambda}^{-}(M) +\mathcal{M}_{\lambda,\Lambda}^{-}(N)\leq \mathcal{M}_{\lambda,\Lambda}^{-}(M+N) \leq \mathcal{M}_{\lambda,\Lambda}^{-}(M) +\mathcal{M}_{\lambda,\Lambda}^{+}(N)$. \\
        \item $ \mathcal{M}_{\lambda,\Lambda}^{+}(M) +\mathcal{M}_{\lambda,\Lambda}^{-}(N)\leq \mathcal{M}_{\lambda,\Lambda}^{+}(M+N) \leq \mathcal{M}_{\lambda,\Lambda}^{+}(M) +\mathcal{M}_{\lambda,\Lambda}^{+}(N).$ 
    \end{enumerate}
\end{prop}
Especially, we have
\begin{align*}
    \mathcal{M}^\pm_{\lambda|z|^p,\Lambda(|z|^p + |z|^q)}(M) = |z|^p\mathcal{M}^\pm_{\lambda,\Lambda(1 + |z|^{q-p})}(M).
\end{align*}
Moreover, we introduce the Ishii-Lions Lemma, (see \cite{Ishii92,Silvestre13}). 
\begin{lem} \label{IL1}
    Let $u,v \in C(\Omega)$ and $\phi(x) \in C^2(\Omega)$.
    Assume that $(\overline{x},\overline{y}) \in \Omega\times\Omega$ is a local maximum points of $u(x)-v(y) -\phi(x-y)$.
    Then for any $\epsilon>0$ such that $\epsilon Z<I$, there exist $X,Y \in S(n)$ such that
    \begin{align*}
        (z,X) \in \overline{\mathcal{J}^{2,+}_\Omega} u(\overline{x}), \quad (z,Y) \in \overline{\mathcal{J}^{2,-}_\Omega} v(\overline{y})
    \end{align*}
    and
    \begin{align*}
        -\frac{2}{\epsilon} \begin{pmatrix}
             I & 0 \\
             0 & I
         \end{pmatrix}
         \leq 
        \begin{pmatrix}
             X & 0 \\
             0 &-Y
         \end{pmatrix}
         \leq 
         \begin{pmatrix}
             Z^\epsilon & -Z^\epsilon \\
             -Z^\epsilon & Z^\epsilon
         \end{pmatrix},
    \end{align*}
    where $z = D\phi(\overline{x}-\overline{y})$, $Z = D^2\phi(\overline{x}-\overline{y})$ and $Z^\epsilon = (I-\epsilon Z)^{-1}Z$.
\end{lem}
For any $-1<p\leq q$, we show that it is always possible to change the $(p,q)$-growth problem into the $(0,\gamma)$-growth problem where $\gamma=q-p\geq0$.
By dividing $|Du|^{p}$ by \eqref{PDE}, we have
\begin{align*}
    \tilde{F}(D^2u,Du,x):=|Du|^{-p}F(D^2u,Du,x)=|Du|^{-p}f(x,Du) :=\tilde{f}(x,Du)
\end{align*}
Then $\tilde{F}$ has $(0,\gamma)$-growth and satisfies all the assumptions $\ref{A1}-\ref{A4}$.
Notably, for any $|z|,|w| >C_0$ with $\frac{1}{2}|w|\leq |z| \leq 2|w|$,
\begin{align*}
    |\tilde{F}(M,z,x)-\tilde{F}(M,w,x)| &\leq |z|^{-p}|F(M,z,x)-F(M,w,x)| + ||z|^{-p}-|w|^{-p}||F(M,w,x)| \\
    &\leq|z|^{-p} \Lambda |z-w|(|z|^{q-1}+|w|^{q-1})(1+\norm{M}) + C|z-w||z|^{-p-1} |w|^q(1+\norm{M}) \\
    &\lesssim \Lambda|z-w|(|z|^{\gamma-1}+|w|^{\gamma-1})(1+\norm{M}).
\end{align*}
Moreover, we also get
\begin{align*}
    |\tilde{f}(x,z)| \leq C_f(|z|^{-p}+|z|^{q-p+1}) 
    \leq C_f(1+|z|^{\gamma+1})
\end{align*}
even for $-1<p\leq0$ since $-p<1\leq q-p+1$.
Therefore, from now on we always consider the $(0,\gamma)$-growth problem with $\gamma=q-p$.

Finally, we consider the scaling property of the $(0,\gamma)$-growth problem to verify that we can assume that $\norm{u}_{L^\infty} \leq 1$ for the $u \in L^\infty$ case and that $\norm{u}_{C^\beta} \leq 1$ for the $u \in C^\beta$ case with different $\Lambda$ and $C_f$.
If $u \in C(B_1)$ is a solution of the $(0,\gamma)$-growth problem \eqref{PDE}, we consider
$$v(x) := \frac{u(rx)}{K}$$
for some $0<r\leq1\leq K$.
Then $v$ is a solution of 
\begin{align*}
    \tilde{F}(D^2v,Dv,x)=\tilde{f}(x,Du),
\end{align*}
where \begin{align} \label{scaledPDE}
    \tilde{F}(M,z,x) := \frac{r^2}{K}F(\frac{K}{r^2}M,\frac{K}{r}z,rx), \quad \tilde{f}(x,z) :=\frac{r^2}{K}f(rx,\frac{K}{r}z).
\end{align}
Note that $\tilde{F}$ and $\tilde{f}$ satisfy the same assumptions \ref{A1}-\ref{A4} with $\Lambda$ and $C_f$ changed to $\tilde{\Lambda} = \left(\frac{K}{r}\right)^\gamma\Lambda$ and $\tilde{C_f}= \left(\frac{K}{r}\right)^\gamma rC_f$, respectively.
Especially, by setting $r=1$ and $K = (1+ \norm{u}_{L^\infty})$, we have $\norm{v}_{L^\infty} \leq 1$ and $v$ satisfies the $(0,\gamma)$-growth equation with $\tilde{\Lambda} = (1+\norm{u}_{L^\infty})^\gamma \Lambda$ and $\tilde{C_f}= (1+\norm{u}_{L^\infty})^\gamma C_f$.
Similarly, we can assume $\norm{v}_{C^\beta} \leq 1$ by setting $r=1$ and $K = (1+ \norm{u}_{C^\beta})$ with different $\tilde{\Lambda}$ and $\tilde{C_f}$, depending on $\norm{u}_{C^\beta}$.
\section{The interior Lipschitz estimates with $(p,q)$ growth} \label{sec3}
In this section, we first prove Theorem \ref{Main1}. The proof is based on the Ishii-Lions method.
\begin{proof}[Proof of Theorem \ref{Main1}.]
    We prove the theorem either $u$ is $C^\beta$ for some $\beta<1$ or $u$ is just continuous.
    By using the scaling above, we assume $\norm{u}_{L^\infty(B_1)} \leq 1$ or $\norm{u}_{C^\beta(B_1)} \leq 1$  but $\Lambda$ and $C_f$ is depending on the `original' $\norm{u}_{L^\infty(B_1)}$ or $\norm{u}_{C^\beta(B_1)}$. 
    We only prove the $u \in C^\beta$ case since for the latter case, the proof is the same as the one given below the proof below with $\beta = 0$. 
    
    First, we show that $u$ is $C^\kappa$ for any H{\"o}lder exponent $\beta<\kappa<1$, and then prove that $u$ is Lipschitz. We fix $x_0 \in B_{1/2}$ and claim that
    \begin{align} \label{intish}
        M:= \max_{x,y\in \overline{B_1}} \left\{ u(x) - u(y) - L_1\phi(|x-y|) - \frac{L_2}{2}|x-x_0|^2 -\frac{L_2}{2}|y-x_0|^2 \right\} \leq 0,
    \end{align}
    where $\phi(t) = t^\kappa/\kappa$ for some $\beta<\kappa<1$ close to 1 and large $L_1,L_2>0$.
    We argue by contradiction by assuming $M>0$ for any large $L_1, L_2>0$.
    From now on, we write $x,y \in \overline{B_1}$ as the points where the maximum $M$ is attained.
    Then we get $x\neq y$ and 
    \begin{align*}
        L_1 \phi(|x-y|) + \frac{L_2}{2}|x-x_0|^2 +\frac{L_2}{2}|y-x_0|^2 \leq |u(x)-u(y)| \leq 2|x-y|^\beta.
    \end{align*}
    We fix large $L_2 > 64$ so that
    \begin{align} \label{xest}
        |x-x_0| < \frac{1}{4}\delta^{\beta/2} \quad \text{ and } \quad  |y-x_0| < \frac{1}{4}\delta^{\beta/2},
    \end{align}
    where $\delta = |x-y| \leq 2$, which implies $x,y \in B_1$.
    Moreover, by using
    \begin{align*}
        L_1 \phi(|x-y|) = \frac{L_1}{\kappa}\delta^\kappa \leq 2\delta^\beta,
    \end{align*}
     we obtain
     \begin{align} \label{dest}
         \delta \leq CL_1^{-\frac{1}{\kappa-\beta}}.
     \end{align}
     Now we apply the Ishii-Lions lemma (Lemma \ref{IL1}) to $u(x)- \frac{L_2}{2}|x-x_0|^2$ and $u(y) + \frac{L_2}{2}|y-x_0|^2$. Then there exist $X,Y \in S(n)$ such that
     \begin{align*}
          (z, X)\in \overline{\mathcal{J}}^{2,+} \left( u(x)- \frac{L_2}{2}|x-x_0|^2\right), \\
          (z, Y)\in \overline{\mathcal{J}}^{2,-} \left( u(y)+ \frac{L_2}{2}|y-x_0|^2\right),
     \end{align*}
     where 
     \begin{align} \label{zest2}
         z= L_1\phi'(|x-y|)\frac{x-y}{|x-y|} = \frac{L_1} {\delta^{1-\kappa}} a \quad \text{and} \quad a= \frac{x-y}{|x-y|}.
     \end{align}
     Thus, we have
     \begin{align*}
         (z_x, X+L_2I)\in \overline{\mathcal{J}}^{2,+} \left( u(x)\right), \\
         (z_y, Y-L_2I)\in \overline{\mathcal{J}}^{2,-} \left( u(y)\right),
     \end{align*}
     where
     \begin{align*}
         z_x = z+L_2(x-x_0) \quad \text{and} \quad z_y = z-L_2(y-x_0).
     \end{align*}
     Note that for large $L_1>0$ depending on $L_2$ and $C_0$, we get
     \begin{align} \label{zxest}
          C_0\leq \frac{1}{2}|z| \leq |z_x|,|z_y| \leq 2|z|.
     \end{align}
     Moreover, for any $\epsilon>0$ such that $\epsilon Z<I$, we can choose $X,Y \in S(n)$ satisfying
     \begin{align} \label{mtx}
     -\frac{2}{\epsilon} 
        \begin{pmatrix}
            I & 0 \\
            0 &I
        \end{pmatrix}
        \leq
        \begin{pmatrix}
             X & 0 \\
             0 &-Y
         \end{pmatrix}
         \leq
         \begin{pmatrix}
             Z^\epsilon & -Z^\epsilon \\
             -Z^\epsilon & Z^\epsilon
         \end{pmatrix},
     \end{align}
     where
     \begin{align*}
         Z &= L_1 \phi''(|x-y|)\frac{x-y}{|x-y|}\otimes\frac{x-y}{|x-y|} + L_1 \frac{\phi'(|x-y|)}{|x-y|} \left(I-\frac{x-y}{|x-y|}\otimes\frac{x-y}{|x-y|}\right) \\
         &=\frac{L_1}{\delta^{2-\kappa}}\left( (\kappa-1)a\otimes a + (I-a\otimes a) \right),
     \end{align*}
     and $Z^\epsilon = (I-\epsilon Z)^{-1} Z$.
     Letting $\epsilon = \frac{\delta^{2-\kappa}}{2L_1}$, we obtain
     \begin{align*}
         Z^\epsilon = L_1\delta^{\kappa-2} \left( \frac{2(\kappa-1)}{3-\kappa}a \otimes a + 2(I-a\otimes a)\right).
     \end{align*}
     Therefore, by \eqref{mtx} we get
     \begin{align} \label{Xest}
         \norm{X}, \norm{Y} \leq 4\frac{L_1}{\delta^{2-\kappa}}.
     \end{align}
     By the definition of viscosity solution, we have
     \begin{align*}
         F(X+L_2I, z_x, x) \geq f(x,z_x) \quad \text{and} \quad F(Y-L_2I, z_y, y) \leq f(y,z_y),
     \end{align*}
     so that
     \begin{align} \label{allest}
     \begin{split}
         -2C_f(1+C|z|^{\gamma+1})\leq f(x,z_x)-f(y,z_y) 
         &\leq F(X+L_2I, z_x, x) -F(Y-L_2I, z_y, y) \\
        &=(F(X+L_2 I ,z_x, x) - F(Y-L_2 I , z_x, x)) \\
        &+ (F(Y-L_2 I , z_x, x) - F(Y-L_2I, z_y, x)) \\
        &+ (F(Y-L_2I, z_y, x) - F(Y-L_2I, z_y, y)) \\
        &=:I_1 + I_2 +I_3.
     \end{split}
     \end{align}
     From now on, we say $C>0$ is a constant depending on $n, p, q, \alpha, \beta, \lambda, \Lambda, C_0,$ and $ C_f$, which may vary from lines to lines.
     We first estimate $I_1$.
     By the assumption \ref{A1}, we have
     \begin{align*}
         I_1 &\leq \mathcal{M}^+_{\lambda, \Lambda(1+|z_x|^{\gamma})}(X-Y+2L_2I) \\
         &\leq \mathcal{M}^+_{\lambda, \Lambda(1+|z_x|^{\gamma} )}(X-Y) +2L_2\mathcal{M}^+_{\lambda, \Lambda(1+|z_x|^{\gamma})}(I).
     \end{align*}
     Applying the inequality \eqref{mtx} to any vector $(b,b)$ with $|b|= 1$, we have $X-Y \leq 0$.
     On the other hand, applying \eqref{mtx} to $(a,-a)$, we have
     \begin{align*}
         \langle(X-Y)a,a\rangle \leq 4\langle Z^\kappa a, a\rangle = -\frac{8L_1}{\delta^{2-\kappa}} \left(\frac{1-\kappa}{3-\kappa}\right) <0.
     \end{align*}
     These two inequalities mean all the eigenvalues of $X-Y$ are non positive and at least one eigenvalue is less than $-\frac{8L_1}{\delta^{2-\kappa}} \left(\frac{1-\kappa}{3-\kappa}\right)$,
     which implies
     \begin{align*}
         \mathcal{M}^+_{\lambda, \Lambda(1+|z_x|^{\gamma})}(X-Y) \leq -C\frac{L_1}{\delta^{2-\kappa}}.
     \end{align*}
     Thus, we derive that
     \begin{align*}
         I_1 \leq -C\frac{L_1}{\delta^{2-\kappa}} + C(1+|z|^\gamma).
     \end{align*}
     Now we estimate $I_2$.
     We use the assumption \ref{A2}, \eqref{xest}, \eqref{zxest} and \eqref{Xest} to obtain
     \begin{align*}
         I_2 &\leq \Lambda|z_x-z_y|(|z_x|^{\gamma-1}+|z_y|^{\gamma-1}) \norm{Y-L_2I} \\
         &\leq C|(x-x_0)+(y-x_0)||z|^{\gamma-1}\norm{Y-L_2I} \\
         &\leq C\delta^{\beta/2}|z|^{\gamma-1}\frac{L_1}{\delta^{2-\kappa}}.
     \end{align*}
     Finally, we estimate $I_3$.
     By the assumption \ref{A3}, \eqref{zxest} and \eqref{Xest}, we have
     \begin{align*}
         I_3 &\leq \Lambda|x-y|^\alpha|z_y|^\gamma \norm{Y-L_2I} \\
         &\leq C\delta^\alpha |z|^\gamma\frac{L_1}{\delta^{2-\kappa}}.
     \end{align*}
     Gathering the previous estimates, we have
     \begin{align*}
         \frac{L_1}{\delta^{2-\kappa}} \leq C\left(|z|^{\gamma+1}+1 + \delta^{\beta/2}|z|^{\gamma-1}\frac{L_1}{\delta^{2-\kappa}} + \delta^\alpha |z|^\gamma\frac{L_1}{\delta^{2-\kappa}}\right),
     \end{align*}
     or equivalently,
     \begin{align*}
         1 \leq C\left(|z|^{\gamma+1} \frac{\delta^{2-\kappa}}{L_1} + \delta^{\beta/2}|z|^{\gamma-1} + \delta^\alpha |z|^{\gamma}\right).
     \end{align*}
     By using \eqref{zest2}, we obtain
     \begin{align*}
         1 \leq C\left( \delta^{2-\kappa -(\gamma+1)(1-\kappa)}L_1^{\gamma} + \frac{\delta^{\beta/2 + (1-\gamma)(1-\kappa)}}{L_1^{1-\gamma}} + \delta^{\alpha - \gamma(1-\kappa)}L_1^\gamma \right).
     \end{align*}
     Notice that when $\kappa <1$ is close enough to 1, the exponents of $\delta$ are positive.
     Thus, using \eqref{dest} we get
     \begin{align*}
         1 \leq C\left( L_1^{-\frac{1}{\kappa-\beta}\left(1-\gamma(1-\beta)\right) }+L_1^{-\frac{1}{\kappa-\beta}\left(\frac{\beta}{2} + (1-\gamma)(1-\beta)\right)} + L_1^{-\frac{1}{\kappa-\beta}(\alpha - \gamma(1-\beta))} \right).
     \end{align*}
     If $\gamma$ satisfies
     \begin{align*}
         \gamma < \frac{1}{1-\beta}, \ \ \gamma < 1+\frac{\beta}{2(1-\beta)}, \ \ \gamma <\frac{\alpha}{1-\beta},
     \end{align*}
     then the exponents of $L_1$ are negative, which makes a contradiction for large $L_1>0$.
     Therefore, $u$ is $C^\kappa$ for any $\kappa<1$ close to 1.
     
     Now we prove $u$ is Lipschitz by using a similar computation.
     As before, we claim that 
     \begin{align*}
        M:= \max_{x,y\in \overline{B_1}} \left\{ u(x) - u(y) - L_1\phi(|x-y|) - \frac{L_2}{2}|x-x_0|^2 -\frac{L_2}{2}|y-x_0|^2 \right\} \leq 0,
    \end{align*}
    where
    \begin{align*}
        \phi(t) = 
        \begin{cases}
            t-\frac{1}{1+\kappa_0}t^{1+\kappa_0} &\text{for } t\in[0,1] \\
            1-\frac{1}{1+\kappa_0} &\text{for } t>1,
        \end{cases}
    \end{align*}
    for some small $0<\kappa_0<1$ to be chosen later.
    We get $x \neq y$ and since $u \in C^\kappa$ for any $\kappa<1$,
    \begin{align*}
        L_1 \phi(|x-y|) + \frac{L_2}{2}|x-x_0|^2 +\frac{L_2}{2}|y-x_0|^2  \leq C\delta^{\kappa},
    \end{align*}
    where $\delta = |x-y|$. Also, fix large enough $L_2>1$ so that 
    \begin{align}
        |x-x_0| < \frac{1}{4}\delta^{\kappa/2} \quad \text{ and } \quad  |y-y_0| < \frac{1}{4}\delta^{\kappa/2}.
    \end{align}
    Moreover, since $L_1\phi(\delta) \leq L_1 \delta \leq C\delta^\kappa,$ we have
    \begin{align} \label{dest2}
        \delta \leq CL_1^{-\frac{1}{1-\kappa}}.
    \end{align}
    Now we again apply the Ishii-Lions lemma (Lemma \ref{IL1}) to $u(x)- \frac{L_2}{2}|x-x_0|^2$ and $u(y) + \frac{L_2}{2}|y-x_0|^2$, to discover that there exist $X,Y \in S(n)$ such that
     \begin{align*}
         (z_x, X+L_2I)\in \overline{\mathcal{J}}^{2,+} \left( u(x)\right), \\
         (z_y, Y-L_2I)\in \overline{\mathcal{J}}^{2,-} \left( u(y)\right),
     \end{align*}
     where
     \begin{align*}
         z_x = z+L_2(x-x_0) \quad \text{and} \quad z_y = z-L_2(y-x_0),
     \end{align*}
    with
     \begin{align} \label{zest}
         z= L_1\phi'(|x-y|)\frac{x-y}{|x-y|} = L_1(1-\delta^{\kappa_0}) a \quad \text{and} \quad a= \frac{x-y}{|x-y|}.
         \end{align}
     Note that for large $L_1$ depending on $L_2$ and $C_0$, then $|z_x|$ and $ |z_y|$ are comparable with $|z|$ and larger than $C_0$.
     Moreover, for any sufficiently small $\epsilon< 1$, we can choose $X,Y \in S(n)$ satisfying
     \begin{align} \label{mtx2}
     -\frac{2}{\epsilon} 
        \begin{pmatrix}
            I & 0 \\
            0 &I
        \end{pmatrix}
        \leq
        \begin{pmatrix}
             X & 0 \\
             0 &-Y
         \end{pmatrix}
         \leq
         \begin{pmatrix}
             Z^\epsilon & -Z^\epsilon \\
             -Z^\epsilon & Z^\epsilon
         \end{pmatrix},
     \end{align}
     where
     \begin{align*}
         Z =\frac{L_1}{\delta}\left( -\kappa_0\delta^{\kappa_0} a\otimes a + (1-\delta^{\kappa_0})(I-a\otimes a) \right),
     \end{align*}
     and $Z^\epsilon = (I-\epsilon Z)^{-1} Z$.
     Letting $\epsilon = \frac{\delta}{2L_1}$, we obtain
     \begin{align*}
         Z^\epsilon = \frac{L_1}{\delta} \left( -\frac{2\kappa_0\delta^{\kappa_0}}{2+\kappa_0\delta^{\kappa_0}}a \otimes a + \frac{2(1-\delta^{\kappa_0})}{1+\delta^{\kappa_0}}(I-a\otimes a)\right).
     \end{align*}
     Therefore, by \eqref{mtx2} we get
     \begin{align} \label{Xest2}
         \norm{X}, \norm{Y} \leq 4\frac{L_1}{\delta}.
     \end{align}
     We use the same computation as in \eqref{allest} and estimate three terms $I_1$, $I_2$ and $I_3$.
     By the inequality \eqref{mtx}, we have $X-Y \leq 0$ and
     \begin{align*}
         \langle(X-Y)a,a\rangle \leq 4\langle Z^\epsilon a, a\rangle \leq -\frac{8\kappa_0}{2+\kappa_0\delta^{\kappa_0}}\frac{L_1}{\delta^{1-\kappa_0}}<0.
     \end{align*}
     Therefore, we get
     \begin{align*}
         \mathcal{M}^+_{\lambda, \Lambda(1+|z_x|^{\gamma} )}(X-Y) \leq -C\frac{L_1}{\delta^{1-\kappa_0}}.
     \end{align*}
     Thus, we derive that
     \begin{align*}
         I_1 &\leq \mathcal{M}^+_{\lambda, \Lambda(1+|z_x|^{\gamma})}(X-Y+2L_2I) \\
         &\leq \mathcal{M}^+_{\lambda, \Lambda(1+|z_x|^{\gamma})}(X-Y) +2L_2\mathcal{M}^+_{\lambda, \Lambda(1+|z_x|^{\gamma})}(I)\\
         &\leq -C\frac{L_1}{\delta^{1-\kappa_0}} + C(1+|z|^\gamma).
     \end{align*}
     For $I_2$, by using \eqref{Xest2} we obtain
     \begin{align*}
         I_2 &\leq \Lambda|z_x-z_y|(|z_x|^{\gamma-1}+|z_y|^{\gamma-1}) \norm{Y-L_2I} \\
         &\leq C|(x-x_0)+(y-x_0)||z|^{\gamma-1}\norm{Y-L_2I} \\
         &\leq C\delta^{\kappa/2}|z|^{\gamma-1}\frac{L_1}{\delta}.
     \end{align*}
     Finally, for $I_3$, we get
     \begin{align*}
         I_3 &\leq \Lambda|x-y|^\alpha|z_y|^\gamma \norm{Y-L_2I} \\
         &\leq C\delta^\alpha |z|^{\gamma}\frac{L_1}{\delta}.
     \end{align*}
     Gathering the previous estimates, we have
     \begin{align*}
         \frac{L_1}{\delta^{1-\kappa_0}} \leq C\left(|z|^{\gamma+1}+1 + \delta^{\kappa/2}|z|^{\gamma-1}\frac{L_1}{\delta} + \delta^\alpha |z|^{\gamma}\frac{L_1}{\delta}\right),
     \end{align*}
          or equivalently,
     \begin{align*}
         1 \leq C\left(|z|^{\gamma+1} \frac{\delta^{1-\kappa_0}}{L_1} + \delta^{\kappa/2-\kappa_0}|z|^{\gamma-1} + \delta^{\alpha-\kappa_0} |z|^{\gamma}\right).
     \end{align*}
     By using \eqref{zest}, we obtain
     \begin{align*}
         1 \leq C\left( \delta^{1-\kappa_0 }{L_1^{\gamma}} + \frac{\delta^{\kappa/2-\kappa_0}}{L_1^{1-\gamma}} + \delta^{\alpha-\kappa_0} L_1^{\gamma} \right).
     \end{align*}
     We choose $\kappa_0<1$ small so that the exponents of $\delta$ are positive.
     Thus, using \eqref{dest2} we get
     \begin{align*}
         1 \leq C\left( L_1^{-\frac{1}{1-\kappa}\left((1-\kappa_0) -\gamma(1-\kappa)\right) }+L_1^{-\frac{1}{1-\kappa}\left(\frac{\kappa}{2} - \kappa_0 + (1-\gamma)(1-\kappa)\right)} + L_1^{-\frac{1}{1-\kappa}(\alpha - \kappa_0 - \gamma(1-\kappa))} \right).
     \end{align*}
     If we choose $\kappa$ close to 1, then the exponents of $L_1$ are negative, which makes a contradiction for large $L_1$.
     Therefore, $u$ is Lipschitz continuous.
\end{proof}
\section{The boundary and global Lipschitz estimates with $(p,q)$ growth} \label{sec4}
In this section, we obtain the boundary and global regularity results for the $C^{1,1}$ domain $\Omega$.
We assume that $0\in \partial\Omega$ and there exists a ball $B=B_R$ and $\phi \in C^{1,1}(\mathbb{R}^{n-1})$ such that $\phi(0)=0$, $D\phi(0) = 0$ and
\begin{align*}
    B\cap \Omega  \subset \{y\in B : y_n>\phi(y')\}, \quad B \cap \partial\Omega = \{ y \in B: y_n=\phi(y')\}.
\end{align*}
For simplicity, we assume that $B=B_1$ is the unit ball in the lemmas below.
\begin{lem} \label{bar}
    Let $d$ be the distance to $\{y_n = \phi(y')\}$, $g\in C^{1,1}(B_1\cap\partial\Omega)$ and $u \in C(B_1\cap\overline{\Omega })$ be a solution of the $(0,\gamma)$-growth problem
    \begin{align*}
        \begin{cases}
        F(D^2u,Du,x) = f(x,Du) \quad &\text{in } \ B_1\cap \Omega \\
        u=g\quad &\text{on } \ B_1\cap \partial \Omega.
        \end{cases}
    \end{align*} 
    If $\norm{u}_{L^\infty(B_1\cap\overline{\Omega})} \leq 1$ and $0\leq\gamma < 1$, then for any $\gamma_0 \in (0, 1-\gamma)$, there exists a small positive constant $\delta_0=\delta_0(n, \lambda, \Lambda, p, q, \gamma_0,C_0,C_f, \norm{g}_{C^{1,1}},\Omega)>0$ such that for any $\delta\leq\delta_0$, there holds
    \begin{align*}
        |u(y',y_n)-g(y')| \leq \frac{8}{\delta}\frac{d(y)}{1+d(y)^{\gamma_0}} \quad \text{in } \ B_{1/2} \cap\Omega \cap \{d(y)<\delta\}.
    \end{align*}
    Moreover, if $\norm{u}_{C^\beta(B_1\cap\overline{\Omega})} \leq 1$ and $0\leq\gamma < \frac{1}{1-\beta}$, then for any $\gamma_0 \in (0, 1-\gamma(1-\beta))$, there exists $\delta_0=\delta_0(n, \lambda, \Lambda, p, q, \beta, \gamma_0, C_0, C_f, \norm{g}_{C^{1,1}},\Omega)>0$ such that for any $\delta\leq \delta_0$, there holds
    \begin{align*}
        |u(y',y_n)-g(y')| \leq \frac{8}{\delta^{1-\beta}}\frac{d(y)}{1+d(y)^{\gamma_0}} \quad \text{in } \ B_{1/2} \cap\Omega \cap \{d(y)<\delta\}.
    \end{align*}
\end{lem}
\begin{proof}
    We give the proof of the $u \in C^\beta$ case only.
    For the $u \in L^\infty$ case, it is enough to put $\beta=0$.
    We consider $\Omega_{\delta}:=\{y \in \Omega : d(y) <\delta\}$ with $\delta<\delta_1$ such that if $d(y) <\delta_1$, then the distance function $d$ is $C^2$ and $|D^2d| <K$ for some constant $K>0$.
    We consider an extension of $g$ to $B_1\cap\Omega$ (still denoted by $g$), which is $C^2$ locally in $B_1\cap\Omega$, with $\|g\|_{C^\beta}\le 1$ and whose $C^{1,1}$ norm is controlled by $\|g\|_{C^{1,1}(B_1\cap\partial\Omega)}$.
    The main goal is to construct an upper barrier $w$ which satisfies
    \begin{align*}
    |Dw|^p\mathcal{M}^+_{\lambda,\Lambda(1 + |Dw|^\gamma)}(D^2w) < -C_f(1+|Dw|^{q+1}) < f(x,Dw) \quad \text{in }B_r\cap\Omega_\delta.
    \end{align*}
    We define $w \in C^2(B_1 \cap \Omega_\delta)$ as
    \begin{align*}
        w(y) = \overline{w}(y) + g(y). 
    \end{align*}
    where (see \cite{Birindelli14})
    \begin{align} \label{barw}
		\overline{w}(y)=
		\begin{cases} 
		\frac{4}{\delta^{1-\beta}}\frac{d(y)}{1+d(y)^{\gamma_0}} & \text{for $|y|<1/2$},\\
		\frac{4}{\delta^{1-\beta}}\frac{d(y)}{1+d(y)^{\gamma_0}}+16(|y|-\frac{1}{2})^3 & \text{for $|y| \geq 1/2$}.
		\end{cases}
	\end{align}
    Note that on a point $y\in \{d=\delta\}$, we can find $x \in \partial\Omega$ with $|x-y| = \delta$. Then since $|u(y)-u(x)|,|g(y)-g(x)| \leq 2\delta^\beta$, we obtain
    \begin{align*}
        w(y) &\geq \frac{4}{\delta^{1-\beta}}\frac{\delta}{1+\delta^{\gamma_0}} + g(y) \\ &\geq 4\delta^\beta -|g(y)-g(x)| - |u(x)-u(y)| + u(y) \geq u(y).
    \end{align*}
    Observe that on $y \in \partial B_1\cap\{d<\delta\}$, $w(y) \geq 16(1-1/2)^3 +g(y)\geq 
 1\geq u(y)$. Note also that on $y \in B_1 \cap \partial\Omega$, $w(y)  \geq g(y) =u(y)$.
    Thus we have $w \geq u$ on $\partial(B_1\cap \Omega_\delta)$. Moreover, we have
    \begin{align*}
        D\overline{w}=
		\begin{cases} 
		\frac{4}{\delta^{1-\beta}}\frac{1+(1-\gamma_0)d^{\gamma_0}}{(1+d^{\gamma_0})^2} Dd & \text{for $|y|<1/2$},\\
		\frac{4}{\delta^{1-\beta}}\frac{1+(1-\gamma_0)d^{\gamma_0}}{(1+d^{\gamma_0})^2}Dd+48\frac{y}{|y|}(|y|-\frac{1}{2})^2 & \text{for $|y| \geq 1/2$},
		\end{cases}
    \end{align*}
    so that 
    \begin{align*}
        C_0\leq\frac{2}{\delta^{1-\beta}}\leq|Dw| \leq \frac{8}{\delta^{1-\beta}} \quad \text{ for small } \delta < 1.
    \end{align*}
    for sufficiently small $\delta < 1$ depending on $\norm{g}_{Lip}$ and $C_0$.
    Moreover, we also get
    \begin{align*}
        D^2w &= -\frac{4\gamma_0 }{\delta^{1-\beta} d^{1-\gamma_0}}\frac{1+\gamma_0+(1-\gamma_0)d^{\gamma_0}}{(1+d^{\gamma_0})^3}Dd\otimes Dd + \frac{4}{\delta^{1-\beta}}\frac{1+(1-\gamma_0)d^{\gamma_0}}{(1+d^{\gamma_0})^2}D^2d +H(y) + D^2g,
    \end{align*}
    where $\norm{H(y)} \leq C$.
    Thus, we obtain
    \begin{align} \label{barcal}
    \begin{split}
        \mathcal{M}^+_{\lambda,\Lambda(1+|Dw|^\gamma)}(D^2w) &\leq-\lambda\frac{4\gamma_0 }{\delta^{1-\beta} d^{1-\gamma_0}} + \Lambda\frac{4}{\delta^{1-\beta}}|Dw|^\gamma \norm{D^2d} + \Lambda|Dw|^\gamma\norm{H(y)+D^2g} \\ 
        &\leq  -\frac{1}{C\delta^{2-\beta-\gamma_0}} + \frac{C}{\delta^{(1-\beta)(\gamma+1)} } + \frac{C}{\delta^{(1-\beta)\gamma}}  \\
        &< -C_f\left(1+\frac{C}{\delta^{(1-\beta)(\gamma+1)}}\right) \\
        & <-C_f(1+|Dw|^{\gamma+1})
        \end{split}
    \end{align}
    in $B_1 \cap \Omega_\delta$ for $\delta$ sufficiently small depending on $\norm{g}_{C^{1,1}}$. Note that we have used the fact that $2-\beta-\gamma_0 >(1-\beta)(1+\gamma)$ for the last inequality. 
    Therefore, we concluded that $ w \geq u$ on $B_1 \cap \Omega_\delta$.
    The lower barrier is easily deduced by considering $-\overline{w}+g$. 
\end{proof}
Moreover, using the lemma above, one can find the proof of the boundary regularity for solutions by using the Ishii-Lions technique in the proof of Theorem \ref{Main1}.
\begin{thm} \label{Bdy1}
    Let $u \in C(B_1\cap\overline{\Omega })$ be a viscosity solution of the $(p,q)$-growth problem
    \begin{align*}
        \begin{cases}
        F(D^2u,Du,x) = f(x,Du) \quad &\text{in } \ B_1\cap \Omega \\
        u=g\quad &\text{on } \ B_1\cap \partial \Omega,
        \end{cases}
    \end{align*}  with $g\in C^{1,1}(B_1\cap\partial\Omega)$.
    If $q-p<\alpha$, then we have $u \in Lip(B_{1/2} \cap \overline{\Omega})$ and
    \begin{align*}
        \norm{u}_{Lip(B_{1/2} \cap \overline{\Omega})} \leq C(\norm{u}_{L^\infty(B_1\cap\overline{\Omega})}+1)^\theta,
    \end{align*}
    for some $\theta= \theta(q-p,\alpha)\geq1$ and $C(n,\lambda, \Lambda, p,q,\alpha,C_0, C_f,\norm{g}_{C^{1,1}(B_1\cap\partial\Omega)},\Omega)>0$.
    
    Moreover, if $u \in C^\beta(B_1\cap\overline{ \Omega})$ with $q-p<\min \left\{1 + \frac{\beta}{2(1-\beta)}, \frac{\alpha}{1-\beta} \right\}$, then we have $u \in Lip(B_{1/2} \cap \overline{\Omega})$ and
    \begin{align*}
        \norm{u}_{Lip(B_{1/2} \cap \overline{\Omega})} \leq C(\norm{u}_{C^\beta(B_1\cap\overline{\Omega})}+1)^\theta,
    \end{align*}
    for some $\theta=\theta(q-p,\alpha,\beta)\geq1$ and $C(n,\lambda, \Lambda, p,q,\alpha,\beta,C_0, C_f,\norm{g}_{C^{1,1}(B_1\cap\partial\Omega)},\Omega)>0$.
\end{thm}
\begin{proof}[Sketch of proof]
    The proof is similar to the proof of Theorem \ref{Main1}, so we mention only where the boundary makes difference.
    We fix $x_0 \in B_{1/2} \cap \overline{\Omega}$ and claim that
    \begin{align*}
        M:= \max_{x,y\in \overline{B_1\cap \Omega}}  \left\{ u(x) - u(y) - L_1\phi(|x-y|) - \frac{L_2}{2}|x-x_0|^2 -\frac{L_2}{2}|y-x_0|^2 \right\} \leq 0.
    \end{align*}
    We argue by contradiction by assuming $M>0$ for any large $L_1, L_2$.
    Let $x,y \in \overline{B_1\cap \Omega}$ be the points where the maximum $M$ is attained.
    Notice that by choosing large $L_2$ we have $x,y \in B_{3/4} $.
    We claim that $x,y \notin \partial \Omega$ for large $L_1$.
    By Lemma \ref{bar}, there exists $M_0$ and $\delta_0$ such that for any $y \in B_{3/4} \cap \Omega \cap \{d(y)<\delta_0\}$, we have
    \begin{align*}
        |u(y)-g(y')| \leq M_0d(y).
    \end{align*}
    By choosing large $L_1$, we have $|x-y| <\delta_0$ as in \eqref{dest} or \eqref{dest2}. If $x \in \partial\Omega$, then since $d(y) \leq |x-y|<\delta_0$, we find
    \begin{align*}
        |u(y)-u(x)| &\leq |u(y)-g(y')| + |g(y')-g(x')| \\
        &\leq M_0d(y) + \norm{g}_{Lip}|y'-x'|\\
        &\leq (M_0 + \norm{g}_{Lip})|x-y|.
    \end{align*}
    Thus by choosing large $L_1$ depending on $M_0$ and $\norm{g}_{Lip}$, we have $M\leq0$, which is a contradiction.
    The rest of the proof is to apply the Ishii-Lions method, as seen in the proof of Theorem \ref{Main1}.
\end{proof}

However, using the interior Ishii-Lions method requires a more restrictive gap condition for the $u \in C^\beta$ case.
Instead, using the Lemma \ref{bar} and the global Ishii-Lions method, we get the global regularity for $u \in C^\beta(\overline{\Omega})$ with an optimal gap condition $q-p < \frac{\alpha}{1-\beta}$.

\begin{proof}[Proof of Theorem \ref{Main2}.]
For $u \in L^\infty$, it is standard to prove the Lipschitz regularity for the global case by using Theorem \ref{Main1} and Theorem \ref{Bdy1}.
Therefore we focus on the case $u \in C^\beta$.
By the scaling, we assume $\norm{u}_{C^\beta(\overline{\Omega})} \leq 1$.
As in the proof of the interior case, we show that $u \in C^\kappa$ for any $\beta<\kappa<1$ and then prove the Lipschitz regularity. We claim that 
    \begin{align} \label{gloish}
        M:= \max_{x,y\in \overline{\Omega}} \left\{ u(x) - u(y) - \frac{L}{\kappa}|x-y|^\kappa \right\} \leq 0
    \end{align}
for some large $L$.
We argue by contradiction by assuming $M>0$ for any large $L$.
Let $x,y \in \overline{\Omega}$ be the points where the maximum is attained.
Then we have $\delta=|x-y|\neq0$ and
\begin{align} \label{dest3}
    \delta \leq L^{-\frac{1}{\kappa-\beta}}.
\end{align}
We claim that $x,y \notin \partial\Omega$ for large $L$. If $x \in \partial\Omega$, then by changing the coordinates, we may assume $x=0$ and that $\phi\in C^2$ is the graph of $\Omega \cap B_R$ with $\phi(0)=0$ and $D\phi(0) =0$ for some $R>0$ depending on $\Omega$.
According to a scaling argument and Lemma \ref{bar}, there exist $M_0>1$ and $\delta_0>0$ such that for any $y \in B_{R/2}\cap\Omega \cap \{d(y) < \delta_0\}$, we get
\begin{align*}
    |u(y)-u(y',\phi(y))|\leq M_0 d(y).
\end{align*}
Then by choosing large $L$ and using \eqref{dest3}, we have $d(y) \leq |x-y| < \delta_0$ and $y \in B_{R/2}$.
Therefore, we get
\begin{align*}
    |u(y)-u(0)| &\leq |u(y)-u(y',\phi(y))| + |u(y',\phi(y))-u(0)| \\
    &\leq M_0d(y,\partial\Omega) + \norm{g}_{Lip}(|y'| +|\phi(y)|) \\
    &\leq M_0|y| + \norm{g}_{Lip}(1+\norm{\phi}_{Lip})|y|.
\end{align*}
Thus $M\leq0$ by choosing $L$ large enough, which makes a contradiction so that $x,y \notin \partial\Omega$.
Therefore applying the Ishii-Lions lemma (Lemma \ref{IL1}) to $u(x)$ and $u(y)$, there exist $X,Y \in S(n)$ such that
\begin{align*}
    (z, X)\in \overline{\mathcal{J}}^{2,+} \left( u(x)\right), \\
    (z, Y)\in \overline{\mathcal{J}}^{2,-} \left( u(y)\right),
\end{align*}
where
\begin{align} \label{zest3}
     z= \frac{L}{\delta^{1-\kappa}} a \quad \text{and} \quad a= \frac{x-y}{|x-y|}.
\end{align}
Moreover, as in the proof of the Theorem \ref{Main1}, we can choose $X,Y \in S(n)$ satisfying
\begin{align} \label{Xest3}
\begin{split}
    \norm{X}, \norm{Y} &\leq 4\frac{L}{\delta^{2-\kappa}}, \\
    \mathcal{M}^+_{\lambda, \Lambda(1+|z|^{\gamma})}(X-Y) &\leq -C\frac{L}{\delta^{2-\kappa}}.
\end{split}
\end{align}
By the definition of viscosity solution, we have
\begin{align*}
    F(X, z, x) \geq f(x,z) \quad \text{and} \quad F(Y, z, y) \leq f(y,z),
\end{align*}
so that
\begin{align*}
    -2C_f(1+|z|^{\gamma+1}) \leq f(x,z)-f(y,z) &\leq F(X, z, x) -F(Y, z, y) \\
        &=(F(X ,z, x) - F(Y, z, x)) \\
        &+ (F(Y, z, x) - F(Y, z, y)) \\
        &=:I_1 + I_3.
\end{align*}
Observe that the difference of gradient term `$I_2$' does not appear and so the assumption \ref{A2}  is not necessary.
Now using the assumption \ref{A1}, \ref{A3} and \eqref{Xest3}, we find
\begin{align*}
    I_1 &\leq \mathcal{M}^+_{\lambda, \Lambda(1+|z|^{\gamma})}(X-Y) \leq -C \frac{L}{\delta^{2-\kappa}}, \\
    I_3 &\leq \Lambda|x-y|^\alpha |z|^\gamma \norm{Y} \leq C\delta^\alpha|z|^\gamma\frac{L}{\delta^{2-\kappa}}.
\end{align*}
Therefore, we obtain
\begin{align*}
    \frac{L}{\delta^{2-\kappa}} \leq C \left( |z|^{\gamma+1} + \delta^\alpha|z|^\gamma\frac{L}{\delta^{2-\kappa}} \right).
\end{align*}
Using \eqref{zest3}, we get
\begin{align*}
    1 \leq C\left( \delta^{2-\kappa-(\gamma+1)(1-\kappa)}L^{\gamma} +\delta^{\alpha - \gamma(1-\kappa)}L^\gamma\right).
\end{align*}
Note that if $\kappa$ is close to $1$, then the exponents of $\delta$ are positive.
Finally, using \eqref{dest3}, we conclude
\begin{align*}
    1 \leq C \left( L^{-\frac{1}{\kappa-\beta}(1+\gamma(1-\beta))} + L^{-\frac{1}{\kappa-\beta}(\alpha-\gamma(1-\beta))} \right).
\end{align*}
If $\gamma < \frac{\alpha}{1-\beta}$, then the exponents of $L$ are negative, which makes a contradiction for large $L$.
Therefore, $u$ is $C^\kappa$ for any $\kappa<1$.
Observing that the proof of Lipschitz regularity is the same as that for the interior case, Theorem \ref{Main1}, we finish the proof.
\end{proof}
\subsection*{Acknowledgement}
The authors thank the anonymous referees for their valuable comments and suggestions that improved the quality and clarity of the paper. 
S.-S. Byun was supported by Mid-Career Bridging
Program through Seoul National University.
H. Kim was supported by the National Research Foundation of Korea(NRF) grant funded by the Korea government [Grant No. 2022R1A2C1009312].

\bibliographystyle{amsplain}
\bibliography{ref}

\end{document}